\documentclass{amsart}
\usepackage{amssymb}
\usepackage{enumerate}
\usepackage{graphicx}
\usepackage{tabularx}

\def\np{\bigskip\noindent}
\def\nl{\smallskip\noindent}

\def\cl#1{{#1}^{\rm cl}}

\def\op{{\rm op}}
\def\AO{{\mathcal  A}}

\def\A{{\rm A}}
\def\det{{\rm det}}

\def\Br{{\rm {\bf Br}}}

\def\BrM{{\rm {\bf BrM}}}

\def\ADE{{\rm ADE}}
\def\D{{\rm D}}
\def\E{{\rm E}}

\def\Sym{{\rm \Sigma}}

\newtheorem{Thm}{\bf{Theorem}}[section]
\newtheorem{Def}[Thm]{\bf{Definition}}

\newtheorem{Lm}[Thm]{\bf{Lemma}}

\newtheorem{lemma}[Thm]{Lemma}
\newtheorem{Prop}[Thm]{\bf{Proposition}}
\newtheorem{Cor}[Thm]{\bf{Corollary}}

\newtheorem{remark}[Thm]{\bf{Remark}}

\renewcommand{\phi}{\varphi}

\def\Z{{\mathbb Z}}
\def\Q{{\mathbb Q}}
\def\R{{\mathbb R}}

\def\B{{\mathcal B}}

\newcommand{\N}{{\mathbb{N}}}

\newcommand{\eps}{\varepsilon}

\def\a{\alpha}
\def\b{\beta}
\def\c{\gamma}
\def\e{\epsilon}

\def\la{\langle}
\def\ra{\rangle}

\author{Arjeh M. Cohen, Bart Frenk, David B. Wales}
\address{Arjeh M. Cohen\\
Department of Mathematics and Computer Science\\
Technische Universiteit Eindhoven\\
POBox 513\\
5600 MB Eindhoven\\
The Netherlands}
\email{A.M.Cohen@tue.nl}
\address{Bart Frenk\\
}
\email{bart\_frenk@yahoo.com}
\address{David B. Wales\\
Mathematics Department\\
Sloan Lab\\
Caltech\\
Pasadena, CA 91125\\
USA}
\email{dbw@its.caltech.edu}

\title{Brauer Algebras of Simply Laced Type}

\date{\today}

\begin{document}

\begin{abstract}
The diagram algebra introduced by Brauer that describes the centralizer
algebra of the $n$-fold tensor product of the natural representation of an
orthogonal Lie group has a presentation by generators and relations that only
depends on the path graph $\A_{n-1}$ on $n-1$ nodes.  Here we describe an
algebra depending on an arbitrary graph $M$, called the Brauer algebra of type
$M$, and study its structure in the cases where $M$ is a Coxeter graph of
simply laced spherical type (so its connected components are of type
$\A_{n-1}$, $\D_n$, $\E_6$, $\E_7$, $\E_8$).  We determine the representations
and find the dimension.  The algebra is semisimple and contains the group
algebra of the Coxeter group of type $M$ as a subalgebra.  It is a ring
homomorphic image of the Birman-Murakami-Wenzl algebra of type $M$; this fact
will be used in later work determining the structure of the
Birman-Murakami-Wenzl algebras of simply laced spherical type.

\bigskip\noindent
{\sc keywords:}
associative algebra, Brauer algebra,
Brauer diagram, Coxeter group, partially ordered set, root system

\end{abstract}

\maketitle

\section{Introduction}\label{sec:intro}
Let $M$ be a graph. We define the Brauer monoid $\BrM(M)$ to be the monoid
generated by the symbols $r_i$ and $e_i$ for $i$ a node of $M$ and $\delta$,
$\delta^{-1}$ subject to the relations of Table \ref{table:BrauerRels}, where
$\sim$ denotes adjacency between nodes of $M$.  The Brauer algebra $\Br(M)$ of
type $M$ is the monoid algebra $\Z[\BrM(M)]$.  As $\delta$ is in the center of
$\BrM(M)$, it is also an algebra over $\Z[\delta^{\pm1}]$ and will often be
regarded as such.  For $M = \A_{n-1}$, this algebra was introduced not by
generators and relations but in terms of diagrams by Brauer \cite{Bra}.  It
was related to the centralizer algebra of the $n$-th tensor power of the
natural representation of a classical group where $\delta $ is the dimension
of the representation.  The BMW algebras, introduced by Birman \& Wenzl
\cite{BirWen} and Murakami \cite{Mur}, are deformations which play a similar
role for quantum groups and are also a useful tool for introducing Kauffman
polynomials, known from knot theory.  In \cite{CGW}, we introduced BMW
algebras of type $M$ for arbitrary $M$. The results of the present paper will
be of use in our determination of the structure of BMW algebras of type $\D_n$
\cite{CGWTangle,CGWBMW} in much the same way the Brauer algebra of type $\A_n$
was of use for Morton \& Wasserman \cite{MorWas} in the structure
determination of the BMW algebra of type $\A_n$.

\begin{center}
\begin{table}[ht]
\begin{tabular}{|lcl|lcl|}
\hline
label&\quad&relation&label&\quad&relation\\
\hline
($\delta$)&\quad&$\delta$ is central&
($\delta^{-1}$)&\quad&$\delta\delta^{-1} = 1$\\
\hline
&&{for $i$}&&&{for $i$}\\
(RSrr)&\quad&$r_i^2=1$&
(RSer)&\quad&$e_ir_i=e_i$\\
(RSre)&\quad&$r_ie_i=e_i$&
(HSee)\label{e-sq-eq-br}&\quad&$e_i^2=\delta e_i$\\
\hline
&&{for $i\not\sim j$}&&&{for $i\not\sim j$}\\
(HCrr)&\quad&$r_ir_j=r_jr_i$&
(HCer)&\quad&$e_ir_j= r_je_i$\\
(HCee)&\quad&$e_ie_j= e_je_i$&&&\\
\hline
&&{for $i\sim j$}&&&{for $i\sim j$}\\
(HNrrr)&\quad&$r_ir_jr_i=r_jr_ir_j$&
(HNrer)&\quad&$r_je_ir_j=r_ie_jr_i$\\
(RNrre)&\quad&$r_jr_ie_j=e_ie_j$&&&\\
\hline
\end{tabular}
\caption{\label{table:BrauerRels}Brauer relations}
\end{table}
\end{center}

\noindent
A look at (RSrr), (HCrr), and (HNrrr) makes it clear that products of the
$r_i$ belong to a subgroup of the monoid $\BrM(M)$ isomorphic to a quotient of
$W(M)$, the Coxeter group of type $M$. Modding out the ideal generated by all
$e_i$, we see that the subgroup itself is in fact isomorphic to $W(M)$ and
that the $r_i$ form a set of simple reflections.  In particular, the rank
of $\Br(M)$ as a module over $\Z[\delta^{\pm1}]$ is infinite if $M$ is not
spherical in the sense of \cite{Bou}.  This means that, if $\Br(M)$ is
finite-dimensional, its connected components are isomorphic to Coxeter graphs
of type $\A_n$ $(n\ge1)$, $\D_n$ $(n\ge4)$, or $\E_n$ $(n=6,7,8)$, which we
abbreviate to $\ADE$.  Also, if $M$ is disconnected, then $\Br(M)$ is a direct
sum of its connected components.  This explains why we are interested in the
cases where $M\in\ADE$.

\begin{Thm}\label{th:main}
The Brauer algebra of type $M\in \ADE$ over $\Z[\delta^{\pm1}]$ 
is free of dimension as given in Table \ref{table:dims}.
When tensored with $\Q(\delta)$, the algebra is semisimple.
\end{Thm}

\begin{table}[ht]
\begin{center}
\begin{tabular}{|l|r|}
\hline
$\ \ M\ \ $ & $\dim(\Br(M))$ \\
\hline
$\A_n$ & $(n+1)!!$ \\
$\D_n$ & $\ \ (2^n + 1)n!! - (2^{n-1} + 1)n!$ \\
$\E_6$ & $1,440,585$ \\
$\E_7$ & $139,613,625$ \\
$\E_8$ & $53,328,069,225$\\
\hline
\end{tabular}
\end{center}
\caption{\label{table:dims}\textrm{Brauer algebra dimensions}}
\end{table}

\noindent
Here $k!! = 1\cdot 3 \cdots (2k - 1)$, the product of the first $k$ odd
natural numbers.  As the submonoid $\la\delta^{\pm1}\ra$ of $\BrM(M)$
generated by $\delta$ and its inverse is a central subgroup of $\BrM(M)$, the
dimension given by the theorem is equal to the cardinality of the quotient
monoid $\BrM(M)/\la\delta^{\pm1}\ra$.

These assertions, with the precise dimensions for the series $\A_n$ and
$\D_n$, were conjectured before, cf.\ \cite{CGW}.  Treating the algebra by
similar generators and relations for $\A_n$ was done by Birman
 and Wenzl in
\cite{BirWen}.  The Brauer diagram algebra for $\A_n$ has the
stated dimension by \cite{Bra}.  A similar approach for $M = \D_n$ ($n\ge4)$
using diagrams appears in \cite{CGWTangle}.

In this paper, independent
arguments are given which use rewrites of monomials to a certain standard form
for upper bounding the dimensions and constructions of irreducible
representations for lower bounding the dimensions.

We analyze the structure of $\Br(M)$ in great detail.  In order to describe
the results, we recall some notions from \cite{CGW2}. Our standard reference
for Coxeter groups and root systems is \cite{Bou}.  Corresponding to each root
$\a$ (always normalized so that $(\a,\a) = 2$), there is a unique reflection
$r_\a\in W$, and, conversely, each reflection $r$ in $W$ has a unique positive
root $\b$ such that $r = r_\b$.  A set of mutually orthogonal positive roots
corresponds bijectively to a set of commuting reflections in $W$.  The group
$W$ acts on the sets of mutually orthogonal positive roots in a unique way
corresponding to conjugation on the sets of reflections. We consider
$W$-orbits under this action.  A set $B$ of mutually orthogonal positive roots
of $W(M)$ is called {\em admissible} if, whenever $\b_1,\b_2,\b_3$ are
distinct roots in $ B$ and there exists a root $\a$ for which $|(\a,\b_i)|=1$
for all $i$, the positive root of $\pm r_\a r_{\b_1}r_{\b_2}r_{\b_3}\a$ is
also in $B$, cf.\ Lemma \ref{lemma:unique_root} below.  In \cite{CGW2}, a
partial ordering was defined on the $W$-orbit $\B$ of an admissible set of
mutually orthogonal positive roots.  Each such $\B$ has a unique maximal
element $B_0$ in this ordering, called the {\em highest element}, see
\cite[Corollary 3.6]{CGW2}.  The set of nodes $i$ of $M$ with $\a_i$
orthogonal to each element of $B_0$ (notation $\a_i\perp B_0$) is denoted
$C_\B$.  A basis for the Brauer algebras of type $\ADE$ will be found that is
parametrized by triples consisting of an ordered pair of admissible sets of
mutually orthogonal positive roots from the same $W$-orbit $\B$ and an element
of $W(C_\B)$, see Proposition \ref{prop:monomials} and Corollary
\ref{cor:basis} below.  In this light, Theorem \ref{th:main} can be clarified
as follows.

\begin{lemma}\label{lm:counting}
The dimensions of Table \ref{table:dims} are equal to
$\sum_{\B}|\B|^2|W(C_\B)|$, where the summation is over all $W$-orbits $\B$ of
admissible sets of mutually orthogonal positive roots.  All orbits
$\B$ of nonempty admissible sets are listed in Table \ref{table:types2}.
\end{lemma}

\begin{table}[ht]
\begin{center}
\begin{tabular}{|c|cccc|}
\hline
$M$ & $|X|$ & $X^\perp\cap\Phi$ & $C_\B$ & $N_W(X)/C_W(X)$\\
\hline
$\A_n$ & $t$ & $\A_{n - 2t}$ & $\A_{n - 2t}$ & $\Sym_t$ \\
\hline
$\D_n$ & $t$ & $\A_1^t \D_{n - 2t}$ & $\A_1 \D_{n - 2t}$ & $\Sym_t$ \\
$\D_n$ & $2t$ & $\D_{n - 2t}$ & $\A_{n - 2t - 1}$ & $W({\rm B}_t)^*$ \\
\hline
$\E_6$ & $1$ & $\A_5$ & $\A_5$ & $\Sym_1$ \\
$\E_6$ & $2$ & $\A_3$ & $\A_2$ & $\Sym_2$ \\
$\E_6$ & $4$ & $\emptyset$ & $\emptyset$ & $\Sym_4$ \\
\hline
$\E_7$ & $1$ & $\D_6$ & $\D_6$ & $\Sym_1$ \\
$\E_7$ & $2$ & $\A_1 \D_4$ & $\A_1 \A_3$ & $\Sym_2$\\
$\E_7$ & $3$ & $\D_4$ & $\A_2$ & $\Sym_3$ \\
$\E_7$ & $4$ & $\A_1^3$ & $\A_1$ & $\Sym_4$\\
$\E_7$ & $7$ & $\emptyset$ & $\emptyset$ & ${\rm L}(3,2)$\\
\hline
$\E_8$ & $1$ & $\E_7$ & $\E_7$ & $\Sym_1$\\
$\E_8$ & $2$ & $\D_6$ & $\A_5$ & $\Sym_2$\\
$\E_8$ & $4$ & $\D_4$ & $\A_2$ & $\Sym_4$\\
$\E_8$ & $8$ & $\emptyset$ & $\emptyset$ & $2^3{\rm L}(3,2)$\\
\hline
\end{tabular}
\end{center}
\caption{\label{table:types2}\textrm{Nonempty admissible sets $X$ of mutually
orthogonal positive roots.  Each line corresponds to the $W$-orbit of a single
$X$ for each possible choice of $|X|$ indicated in the second column except
for the first line for $\D_{n}$ when $n=2t$, in which case there are two
$W$-orbits with $|X| = n/2$, conjugate by an outer automorphism.  The values
of $t$ lie in $\Z\cap[1,n/2]$.  The third column lists the Cartan type of the
root system on the roots orthogonal to $X$.  The centralizer $C_W(X)$ is the
semi-direct product of the elementary abelian group of order $2^{|X|}$
generated by the reflections in $W$ with roots in $X$ and the subgroup
$W(X^\perp\cap\Phi)$ of $W$ generated by reflections with roots in
$X^\perp\cap\Phi$.  The fourth column lists $C_\B$ for $\B=WX$, the $W$-orbit
of $X$, and the last column lists the structure of $N_W(X)/C_W(X)$.  Here
$W({\rm B}_t)^*$ is understood to be $W({\rm B}_t)$ if $t <n/2$ and $W(\D_t)$
if $t = n/2$.}}
\end{table}

\begin{proof}
See \cite{CGW2} for the second statement (in
[loc.~cit.] the type of $C_\B$ 
for $M=\E_7$ and $|B_0|=2$ is incorrect).

As for
the first statement,
the size of the $W$-orbit $\B$ of an element $X$ from the table
is equal to
$$\frac{|W(M)|}{|N_W(X)|} = \frac{|W(M)|}{ 2^{|X|}|W(X^\perp\cap\Phi)|\cdot
|N_W(X)/C_W(X)|}.$$ All factors occurring in the last expression can be
determined by means of the table and the knowledge of orders of Coxeter groups
of type $\ADE$.  The statement now follows from the following expressions of
the relevant numbers for the individual types.  
For $M=\A_{n-1}$ $(n\ge2)$, the summation
gives
$$\sum_{t=0}^{\lfloor n/2\rfloor} \left(\frac{n!}{2^{t}t!(n-2t)!}\right)^2
(n-2t)! $$ which adds up to $n!!$.  The equality between this summation and
the expression of Table \ref{table:dims} can be proved 
directly as in \cite[p.~113]{Wilf}, or by counting
Brauer diagrams in two different ways, as is clear from \cite{Bra}.
For
$M=\D_n$ $(n\ge4)$, the summation is
$$2^{n-1}n!+ \sum_{t=1}^{\lfloor n/2\rfloor}
\left(\frac{n!}{t!(n-2t)!}\right)^2 2^{n-2t}(n-2t)!  + \sum_{t=1}^{\lfloor
n/2\rfloor} \left(\frac{n!}{2^{t}t!(n-2t)!}\right)^2 (n-2t)! $$ which, by the
formula for $\A_{n-1}$, is easily seen to be $2^{n-1}n! + 2^n( n!!-n!) +
(n!!-n!)$; hence it coincides with the expression for $\D_n$ in Table
\ref{table:dims}.  Here the expression in the first sum over $t$
for $t={n}/{2}$ is in fact a sum over the two orbits.  It is
$2\left(\frac{n!}{2({n}/{2})!}\right)^22$, where
the leftmost $2$ occurs because there are two orbits and the rightmost $2$
accounts for the $\A_1$ component in $C_\B$.
Therefore, the summand becomes
$\left(\frac{n!}{(n/2)!}\right)^2$, and so the expression
$\left(\frac{n!}{t!(n-2t)!}\right)^2 2^{n-2t} (n-2t)!$ is valid for
all $t\le n/2$.  In the second sum, the summand for $t = n/2$ also gives
the expected answer by a deviation from the usual pattern:
the group $C_W(X)$ has order $2^{n-1}(n-2t)!$ and
$N_W(X)/C_W(X)$ has order $2^tt!$ for $t<n/2$, but the respective orders
are $2^n(n-2t)!$ and $2^{t-1}t!$ if $n=2t$
(as the type of the latter is ${\rm D}_{n/2}$ rather than ${\rm B}_{n/2}$), so
$|N_W(X)| = 2^{n+t-1}(n-2t)!t!$ in all cases.
An interpretation in terms of numbers of certain diagrams of type
$\D_n$ will be given in \cite{CGWTangle}.  For $M=\E_6$, the summation is
$|W(\E_6)| + 36^2 |W(\A_5)| + 270^2 |W(\A_2)| + 135^2$,
for $M=\E_7$,
$|W(\E_7)| + 63^2 |W(\D_6)| + 945^2 |W(\A_1\A_3)| + 315^2|W(\A_2)| +
  945^2|W(\A_1)| + 135^2 $,
and for $M=\E_8$,
$|W(\E_8)| + 120^2 |W(\E_7)| + 3780^2 |W(\A_5)| + 9450^2|W(\A_2)| + 2025^2 $.
\end{proof}

After some preliminaries on admissibility in Section \ref{sec:prel} and the
construction of a presentation as maps and a linear representation of the
monoid $\BrM(M)$ in Section \ref{sec:rep}, we prove the upper bound of
$\dim(\Br(M))$ in Section \ref{sec:upb} and the lower bound in Section
\ref{sec:lb}.  At the very end we describe how the Brauer diagrams for type
$\A_{n-1}$ can be extended to a `geometric' picture involving roots for other
types in $\ADE$.  These are in terms of the triples described above Lemma
\ref{lm:counting}.  For those familiar with the Brauer diagrams the triples
may be interpreted as knowledge of the horizontal lines on the top, the
horizontal lines on the bottom, and the permutation of the remaining lines,
see Remarks \ref{rmk:triple1} and \ref{rmk:triple2} below.

Some of the work reported here grew out of the Master's thesis of one of us,
\cite{Frenk}. The other two authors wish to acknowledge Caltech and Technische
Universiteit Eindhoven for enabling mutual visits.

\section{Admissibility}\label{sec:prel}
In this section we mention some basic properties of Brauer algebras related to
admissible sets of mutually orthogonal positive roots that are useful for the
proof of the main theorem.

There exists a notion of root system $\Phi$, \cite{Deo82}, which specializes
to the usual root system for $M\in\ADE$, see \cite{Bou}, and similarly for the
set $\Phi^+$ of positive roots in $\Phi$.  For $\a,\b\in\Phi$, we write
$\a\sim\b$ to denote $|(\a,\b) | = 1$. Thus, for $i$ and $j$ nodes of $M$, we
have $\a_i\sim\a_j$ if and only if $i\sim j$.

\begin{lemma}\label{lemma:unique_root}
Let $M\in\ADE$ and let $\beta_1$, $\beta_2$, $\beta_3$ be three mutually
orthogonal roots of $W(M)$. Then, up to sign, there is at most one $\langle
r_{\b_1},r_{\b_2},r_{\b_3}\rangle$-orbit of roots $\gamma$ with $\gamma\sim
\beta_i$ for $i=1,2,3$. Moreover, there is a unique fourth positive root
$\b_4$ orthogonal to $\b_1$, $\b_2$, $\b_3$ such that, for each such $\c$, we
have $\b_4\sim\gamma$. This root satisfies $\b_4 = \pm r_{\c} r_{\b_1}
r_{\b_2} r_{\b_3} \c$.
\end{lemma}

\begin{proof}
Suppose that $\gamma$ and $\gamma'$ are roots 
with  $\gamma\sim \beta_i\sim \gamma'$ for $i=1,2,3$.
After replacing each $\b_i$ by its negative if needed,
we may assume $(\gamma,\b_i) = -1$ for $i=1,2,3$.
Now
$\b_4 =  r_{\c} r_{\b_1} r_{\b_2} r_{\b_3} \c=
2\c+\b_1+\b_2+\b_3\in\Phi$ is a root orthogonal to $\b_1$, $\b_2$, $\b_3$
with $(\b_4,\c) = 1$.
Also, $\eps= r_\c\b_4 = \c+\b_1+\b_2+\b_3$ is a root.  
Replacing $\gamma'$ by $r_{\b_i}\gamma'$ if
needed for successive values of $i$,
we can arrange for $(\gamma',\beta_i) = -1$ if $i\in \{1,2,3\}$.
If $\gamma'$ does not coincide with $\gamma$,
then $(\gamma',\c)\leq 1$, so
\begin{eqnarray*}
(\gamma',\eps)
& = & (\gamma' , \gamma + \beta_1 + \beta_2 + \beta_3) 
 =  (\gamma' ,\c) - 3\le -2. 
\end{eqnarray*}
The only possibility of this being an integer with norm at most $2$ occurs
when $(\c',\eps) = -2$, that is,
$\gamma' =-\eps = -r_{\b_1} r_{\b_2} r_{\b_3} \c$,
which, up to sign, belongs to 
the $\langle
r_{\b_1},r_{\b_2},r_{\b_3}\rangle$-orbit of $\gamma$.

As for uniqueness of $\b_4$, observe that the linear span of
$\b_1$, $\b_2$, $\b_3$, and $\c$ does not depend on the choice of $\c$
and contains $\b_4$. But in that 4-dimensional space,
$\b_4$ or $-\b_4$ is the unique positive root orthogonal to 
$\b_1$, $\b_2$, and $\b_3$.
\end{proof}

Let $X$ be a set of mutually orthogonal positive roots.  Then, by the lemma,
for each triple of elements in $X$ for which there exists a root $\c$
non-orthogonal to each of the triple, there is a unique element of $\Phi^+$
distinct from $\c$, non-orthogonal to $\c$, and orthogonal to each root from
the triple. Therefore, the intersection of any collection of admissible sets
of mutually orthogonal positive roots is again admissible.  Consequently, the
following notion is well defined as the intersection of all admissible sets
containing $X$.

\begin{Def}\label{admissibleclosure}  
Given a set $X$ of mutually orthogonal positive roots, the unique smallest
admissible set containing $X$ is called the {\em admissible closure} of $X$,
and denoted $\cl{X}$.
\end{Def}

In view of Lemma \ref{lemma:unique_root}, the closure of $X$ can be
constructed by iteratively finding all $\b_1,\b_2,\b_3\in X$ for which there
is a root $\c$ with $\b_i\sim\c$ for all $i$, and 
adjoining the positive root of $\pm r_\c
r_{\b_1}r_{\b_2}r_{\b_3}\c$ to $X$.

\section{Representations of the Brauer Monoid}
\label{sec:rep}

Throughout this section, we assume that $M$ is of type ADE.  Set $W=W(M)$. As
mentioned before the statement of Theorem \ref{th:main}, $W$ occurs as a
subgroup of $\BrM(M)$.  The elements $r_i$ of $\BrM(M)$, for $i$ nodes of $M$,
are a set of simple reflections of $W$, cf.~\cite{Bou}.  It will be convenient
to have more relations for $\Br(M)$ at our disposal than those given in
Table~\ref{table:BrauerRels}.

\begin{table}[ht]
\begin{center}
\begin{tabular}{|lcl|}
\hline
label&\quad&relation\\
\hline
&&for $i\sim j$\\
(RNerr)&\quad&$e_ir_jr_i=e_ie_j$\\
(HNree)&\quad&$r_je_ie_j=r_ie_j$\\
(RNere)&\quad&$e_ir_je_i=e_i$\\
(HNeer)&\quad&$e_je_ir_j=e_jr_i$\\
(HNeee)&\quad&$e_ie_je_i=e_i$\\
\hline
&&for $i\sim j\sim k$\\
(HTeere)&\quad&$e_je_ir_ke_j=e_jr_ie_ke_j$\\
(RTerre)&\quad&$e_jr_ir_ke_j=e_je_ie_ke_j$\\
\hline
\end{tabular}
\caption{\label{AddBrauerTable}Additional relations}
\end{center}
\end{table}

\begin{Lm}\label{additionalrel}  
The relations in Table \ref{AddBrauerTable} also hold in $\Br(M)$.
\end{Lm}

\begin{proof}
For (RNerr), we apply (HNrre), (RSrr), (HNrer), and (RSrr), respectively, to
obtain $e_ie_j=r_jr_ie_j=r_jr_ie_jr_ir_i=r_jr_je_ir_jr_i=e_ir_jr_i$.  For
(HNeee) multiply $e_ir_jr_i=e_ie_j$, from (RNerr), by $ r_ir_j$ and use (RSrr)
and (RNerr) respectively, to get $e_i=e_ie_jr_ir_j=e_ie_je_i$.  For (HNree),
use (RNrre) and (RSrr) to derive $r_je_ie_j =r_jr_jr_ie_j = r_ie_j$.  For
(RNere), use (RSrr), (RSer), (RNrre), and (HNeee) to compute $e_ir_je_i =
e_ir_ir_ir_je_i = e_ie_je_i = e_i$.  For (HNeer), use the reversed words of
(HNree) and notice (RNerr) holds.  For (HTeere), use (RNerr) and (RNrre) to
find $e_je_ir_ke_j = e_jr_ir_jr_ke_j = e_jr_ie_ke_j$.  Finally, for (RTerre),
use (RSrr), (RNerr), and (RNrre) to compute $e_jr_ir_ke_j = e_jr_ir_jr_jr_ke_j
= e_je_ie_ke_j $.
\end{proof}

By $\AO$ we
denote the collection of admissible sets of mutually orthogonal positive
roots.  Let $\B$ be a $W$-orbit in $\AO$. Denote $B_0$ its highest element with
respect to the partial order defined on $\B$; see \cite{CGW2}
for this partial order and the proof
of existence of $B_0$ . The set of nodes $i$ of $M$ for
which $\a_i\perp B_0$ plays an important role in \cite{CGW2}; here it will be
denoted $C_\B$ or, if no confusion is imminent, just $C$. It is well known,
\cite{Bou}, that the subgroup $W(C)$ of $W$ generated by the $r_i$ for $i\in
C$ is a Coxeter group whose type is the restriction of $M$ to $C$.

We present a useful representation of the Brauer monoid as a set of maps from
$\AO$ to itself. At the same time, for each $W$-orbit $\B$ within $\AO$, we
construct a linear representation of the Brauer algebra with basis indexed by
$\B$ and with coefficients from the group ring of $W(C_\B)$ over
$\Z[\delta^{\pm1}]$.  We begin with the action on $\AO$.

\begin{Def}\label{def:Baction} \rm  
Let $\AO$ be the disjoint union of all admissible $W$-orbits (so the empty set
is a member of $\AO$). The action of $W$ on $\AO$ is already given and
corresponds to conjugation on sets of reflections. The action of $\delta$ is
taken to be trivial.  This action extends to an action of the generators $e_i$
of the Brauer monoid in the following way, for $i\in M$ and $B\in \AO$:
\begin{equation}\label{def:sigma_e}
e_iB = \begin{cases}B&\mbox{if  } \a_i\in B,\\
\cl{(B\cup\{\a_i\})}&\mbox{if }\a_i\perp B,\\
r_{\b}r_i B &\mbox{if }\b\in B\setminus\a_i^\perp.
\end{cases}
\end{equation}
\end{Def}

\begin{lemma}\label{lm:admclosure}
For each admissible set $B$, set 
$X$ of mutually orthogonal positive roots
(not-necessarily admissibly closed),
node $i$ of $M$, and positive root $\c$, the
following properties hold.
\begin{enumerate}[(i)]
\item $\a_i\in e_iB$.
\item If $\c\perp X $, then $\c \perp \cl{X}$ or $\c\in \cl{X}$.
\item If $(\a_i,\c)=0$ and $\c\perp B $, then $\c \perp e_iB$ or $\c\in e_iB$.
\item If $w\in W$, then $w\cl{X}=\cl{(wX)}$.
\item The element $e_iB$ is well defined.
\end{enumerate}
\end{lemma}

\begin{proof}
(i) is direct from the definition of the action of $e_i$.
(Observe that $\a_i=\pm r_\b r_i\b$ if $\b\sim \a_i$.)

\nl(ii).  Suppose that $\c_1,\c_2,\c_3$ are roots in $X$ and $\a\in\Phi$ has
inner products $-1$ with each of these. The admissible closure of $X$ will
then contain the positive root $\zeta$ of $\pm(\c_1+\c_2+\c_3+2\a)$, see Lemma
\ref{lemma:unique_root}. If $\c$
is not orthogonal to $\zeta$, then, by the assumption $\c\perp X$, we must
have $0\ne (\c,\zeta) = 2(\c,\a)$. Therefore, $(\c,\a) = \pm 1$
and $(\c,\zeta) = \pm 2$, which means $\c=\pm\zeta$.
As both $\c$ and $\zeta$ are positive, we find $\c=\zeta\in \cl{X}$.

\nl(iii).  If $\a_i\in B$, then $e_iB = B$, and so the conclusion holds by the
hypothesis $\c\perp B$.  If there is $\b\in B\setminus\a_i^\perp$, then $e_iB
= r_\b r_i B$, which consists fully of roots orthogonal to $\c$.

Finally, suppose $\a_i\perp B$. Then $e_i B = \cl{(B\cup \{\a_i\} )}$ and so
the assertion follows from (ii).

\nl(iv). If $\a,\b,\c $ are mutually orthogonal roots joined to $\zeta$, the same is true
for the $w$ images.

\nl(v). 
Ambiguity arises if there are two choices, say $\b$ and $\c$, of roots in
$B\setminus\a_i^\perp$.  We need to show that then $r_\b r_i B = r_\c r_iB$.
Clearly, $r_\b r_i (B\cap \a_i^\perp) = B\cap \a_i^\perp = r_\c r_i(B\cap
\a_i^\perp)$.  For simplicity choose $\b$ and $\c$ so that the inner product
with $\a_i$ is $-1$.  Then $r_\b r_i \c = r_\b(\a_i+\c)=\a_i+\b+\c = r_\c r_i
\b$.  Now $r_\b r_i\{\b , \c\} = \{\a_i, \a_i+\b + \c \} =r_\c r_i\{\b,
\c\}$.

Suppose that $\eta$ is another root in $B\setminus\a_i^\perp$. Then, as $B$ is
admissibly closed, there will be a fourth root $\zeta$ in
$B\setminus\a_i^\perp$. In fact, the fourth is $\zeta = \pm
(\b+\c+\eta+2\a_i)$.  Using this observation it is easily checked that both
$r_\b r_i$ and $r_\c r_i$ leave the set $\{\zeta,\eta\}$ invariant.
\end{proof}

We now define a linear representation of the Brauer algebra.  In \cite{CGW2}
simple reflections $h_{B,i}$ of $W(C_\B)$ were defined for nodes $i$ of
$M$ and members $B$ of $\B$ with $\a_i\perp B$. Extend this definition to all
pairs $(B,i)$ by $h_{B,i} = 1$ if $\a_i\not\perp B$. Let $V_\B$ be the free
right $\Z[\delta^{\pm1}][W(C_\B)]$-module with basis $\xi_B$ for $B\in\B$. For
$B\in\B$ and $i$ a node of $M$, set
\begin{equation}\label{def:rho_i}
r_i \xi_B  =  \xi_{r_iB} h_{B,i} .
\end{equation}

\begin{lemma}\label{lm:rep}  There is a unique linear representation
$\rho_\B: W\to {\rm GL}(V_\B)$ determined by (\ref{def:rho_i}) on the
generators of $W$.
\end{lemma}

\begin{proof}
It is shown in \cite{CGW2} that a similar map is a monoid representation.  The
value of $m$ there can be taken to be $0$ here which simplifies some of the
expressions.  The only difference is that in \cite{CGW2} if $\a_i\in B$, the
image under $r_i$ on $\xi_B$ in our set-up is~$0$.  Here we have
$r_i\xi_B=\xi_B$. Thus, we only treat the cases where this rule applies.

We first discuss the case where $\a_i\in B$.  Here we have
$r_i\xi_B=\xi_B$.  It is immediate that in this case $r_i^2\xi_B=\xi_B$ as
needed.  Suppose $i\not \sim j$ and so $r_i$ and $r_j$ commute.  We must show
$r_ir_j\xi_B=r_jr_i\xi_B$.  Clearly $r_jr_i\xi_B=r_j\xi_B$.  But this is
$\xi_{r_jB}h_{B,j}$ by definition. As $\a_i\in B$, also $\a_i\in r_jB$, for
$r_j\a_i=\a_i$ when $i\not \sim j$.  This means
$r_ir_j\xi_B=r_i\xi_{r_jB}h_{B,j} = \xi_{r_jB}h_{B,j}=r_j\xi_B=r_jr_i\xi_B$
and we are done.  Suppose $i\sim j$.  We need to show
$r_ir_jr_i\xi_B=r_jr_ir_j\xi_B$.  Now by definition
$r_ir_jr_i\xi_B=r_ir_j\xi_{B}$.  As
$\a_j=r_ir_j\a_i \in r_ir_jB$, we also have
$r_jr_ir_j\xi_B =r_ir_j\xi_B$ and we are done.

The only other possibility is that in acting by $r_i$ in the case $i\not \sim
j$ or by $r_ir_j$ in the case $i\sim j$ we would have $\a_j\in r_iB$ in the
first case, or $\a_i\in r_jB$ or $\a_j \in r_ir_jB$ in the second case.  If
$i\not \sim j$ and $\a_i\in r_jB$, then $\a_i\in B$ and we are back in the
previous case. 
Suppose therefore $i\sim j$.  As $\a_j\in r_ir_jB$ implies $\a_i\in B$, it
suffices to consider the case where $\a_i\in r_jB$.  Now $\a_i+\a_j\in B$, so
$\a_i\in r_jB$. Moreover,
$r_jr_ir_j\xi_B=r_jr_i\xi_{r_jB}=r_j\xi_{r_jB}=\xi_B$.  This is symmetric in
$i$ and $j$ and we are done.
\end{proof}

The map $\rho_{\B}$ extends to a representation of $\BrM(M)$. 
The action of $\delta$ is by homothety (so
$\delta v = v\delta$ for $v\in V_\B$). Furthermore, the action of $e_i$ is
defined as follows.
\begin{equation}\label{def:rho_e}
e_i \xi_B  = \left\{
\begin{array}{ll}
\xi_B \delta & \mbox{if $\a_i \in B$} ,\\
0 & \mbox{if $\a_i \perp B$} ,\\
r_\beta r_i  \xi_{B} & \mbox{where $\beta \in B$ and $\beta \sim \a_i$}.
\end{array}
\right.
\end{equation}

Before establishing that this is indeed a representation, we prove that the
action of $e_i$ on $\xi_B$ is well defined.  If $B\in\B$, we will write $K_B$
for the subgroup $\{w\in W\mid w\xi_B = \xi_B\}$ of $W$.  Clearly, $vK_Bv^{-1}
= K_{vB}$ whenever $v\in W$.

\begin{lemma}\label{lm:Hgens}
If $i$ is a node of $M$ and $B\in\B$ has elements $\b$ and $\c$
with $\b\sim \a_i\sim \c$,
then $r_ir_\b r_\c r_i\in K_B$.
\end{lemma}

\begin{proof}
Take $w\in W$ with $w\a_k = \b$ and $w\a_l = \c$ for nodes $k$ and $l$ of $M$.
Such a $w$ always exists.  Now, as $r_i$ moves $B$, we have $r_\b r_i\xi_B =
r_\c r_i\xi_B $ if and only if $r_\b\xi_{r_iB} = r_\c\xi_{r_iB} $, which holds
if and only if $wr_kw^{-1}\xi_{r_iB} = w r_lw^{-1}\xi_{r_iB} $. This is in
turn equivalent to $r_kw^{-1}\xi_{r_iB} = r_lw^{-1}\xi_{r_iB} $, and hence to
$r_k\xi_{w^{-1}r_iB}c = r_l\xi_{w^{-1}r_iB}c $ for some $c\in W(C_\B)$, which
is obviously equivalent to $r_k\xi_{w^{-1}r_iB} = r_l\xi_{w^{-1}r_iB} $.  Set
$B' = w^{-1}r_iB$.  Observe that $w^{-1}r_i\b$, and $w^{-1}r_i\c$ belong to
$B'$ and are moved by $r_k$ and $r_l$. Therefore, $r_k$ and $r_l$ move $B'$,
and so $r_k\xi_{B'} = \xi_{r_kB'}$ and $r_l\xi_{B'} = \xi_{r_lB'}$.  But
$r_kB' = r_k w^{-1}r_iB = w^{-1} r_\b r_i B = w^{-1} r_\c r_i B = r_l w^{-1}
r_i B = r_lB'$, whence $r_k\xi_{w^{-1}r_iB} = r_l\xi_{w^{-1}r_iB} $.
Therefore $r_\b r_i\xi_B = r_\c r_i\xi_B $, as required.
\end{proof}

Consequently, if $\a_i \sim \b,\c\in B$, the two definitions $r_\b r_i\xi_B$
and $r_\c r_i\xi_B$ of $e_i\xi_B$ coincide, so $e_i\xi_B$ is well defined.

For a set $Y$, we write $\mathcal{F}(Y)$ to denote the monoid
of all maps from $Y$ to itself.

\begin{Thm}\label{prop:BrMonAction}  For
each $M\in ADE$, corresponding Coxeter group $W = W(M)$,  and
$W$-orbit $\B$ in $\AO$, the following holds.
\begin{itemize}
\item [(i)]
There is a unique homomorphim $\sigma:
\BrM(M)\to\mathcal{F}(\AO)$ of monoids determined by
the usual action of the generators $r_i$ and
the $e_i$ action of (\ref{def:sigma_e}) on $\AO$.
If $Y$, $X\in\AO$ and $a\in \BrM(M)$ satisfy $Y\subseteq X$, then
$aY\subseteq aX$.
\item[(ii)] There is a unique linear representation, also denoted $\rho_\B$,
of the Brauer algebra $\Br(M)$ on
$V_\B$ extending the map
$\rho_\B$ of Lemma \ref{lm:rep} with $e_i$ acting according to
(\ref{def:rho_e}).
\end{itemize}
\end{Thm}

\begin{proof}
In order to show that $\sigma$ and $\rho_{\B}$ are homomorphisms,
we need to show that they respect the defining relations of $\BrM(M)$.  For
$\sigma$, as the action by $W$ is a group action, and for $\rho_{\B}$, as the
restriction to $W$ is a group representation by Lemma~\ref{lm:rep}, we only
need check the relations for $\BrM(M)$ involving $e_i$'s.
We check both parts at the same time for each of 
these relations in Table \ref{table:BrauerRels}.

We abbreviate $C_\B$ to $C$.
On several occasions, we will use the observation
that, if $e_i\xi_B\ne 0$, then $e_i\xi_B \in\xi_{B'}W(C)$
for $B'\in \B$ with $\a_i\in B'$.  We will then write $e_i\xi_B=\xi_{B'}h$
with $h\in W(C)$.  

\nl (RSer).  For (i) we need to verify $e_ir_iB = e_iB$.  If $\a_i\in B$ or
$\a_i\perp B$, then $r_iB = B$, so we are done.  Suppose, therefore, that
there is $\b\in B\setminus\a_i^\perp$.  Then $r_i\b\in r_i(B\setminus
\a_i^\perp)$ and $\a_i\in e_iB$.  This implies $r_ie_iB=e_iB$.  Now $e_iB =
r_i(e_iB) = r_i r_\b r_i B = r_{r_i\b} r_i (r_i B)=e_ir_iB$, as required.

For (ii), the representation, we must show $e_ir_i\xi_B=e_i\xi_B$.  If
$\a_i\in B$, then $r_i\xi_B = \xi_B$, so we are done.  If $\a_i\perp B $ then
$r_i\xi_B=\xi_Bh_{B,i}$.  But $e_i\xi_B=0$ so both sides are $0$.  Suppose,
therefore, that there is $\b\in B\setminus\a_i^\perp$.  Then $e_i\xi_B=r_\b
r_i\xi_B=r_\b \xi_{r_iB}$.  Now $r_iB$ contains $r_i\b $ which is not
perpendicular to $\a_i$.  Hence $e_ir_i\xi_B=e_i \xi_{r_i B}
=r_{r_i\b}r_i\xi_{r_iB}=r_ir_\b\xi_{r_iB}=r_ir_\b r_i\xi_B$.  Notice that
$\a_i\in r_\b r_iB$ and so there is $h\in W(C)$ such that $e_ir_i\xi_B=r_ir_\b
r_i\xi_B=r_i\xi_{r_\b r_iB}h=\xi_{r_\b r_iB}h= r_\b r_i\xi_B=e_i\xi_B$, as
required.

\nl (RSre).  Here, for (i), we need to show
$r_i e_iB = e_iB$. As $\a_i\in e_i B$, this is immediate.

For (ii) we need to show
$r_i e_i\xi_B = e_i\xi_B$. If $\a_i\in B$, both sides are equal to $\xi_B$ and
if $\a_i\perp B$, both sides are equal to $0$.
Suppose $\b\in B$ 
is not perpendicular and not equal to $\a_i$.
Now $e_i\xi_B=r_\b r_i\xi_B=\xi_{r_\b r_iB}h$ for some $h\in W(C)$.  As
$\a_i\in r_\b r_iB$,
the reflection $r_i$ fixes $\xi_{r_\b r_iB}$
and so $r_i e_i\xi_B=r_i\xi_{r_\b r_iB}h = 
\xi_{r_\b r_iB}h=  e_i\xi_B$, as required.

\nl (HSee). For (i) we need to derive $e_i e_iB =  e_iB$.
As $\a_i\in e_i B$, this is immediate.

For (ii) we need $e_i e_i\xi_B= \delta e_i\xi_B$.  If $\a_i\in B$ or
$\a_i\perp B$, this is immediate. Otherwise $e_i\xi_B=\xi_{B'}h$ with $h\in
W(C)$ and $B'\in \B$ containing $\a_i$, and so the equality follows from
$e_i\xi_{B'}=\delta \xi_{B'}$.

\nl (HCer).  Here $i\not\sim j$. For (i) we need to show $e_i r_jB = r_je_iB$.
If $\a_i\in B^\perp$, then $\a_i$ is also in $(r_jB)^\perp$ and so the result
in both cases is the closure of $r_jB\cup\{ \a_i\}$.  If $\a_i \in B$, then
$r_j\a_i=\a_i$ and so $\a_i\in r_jB$.  Now $e_ir_jB=r_jB$ and $r_je_iB=r_jB$.
Now suppose there is $\b\in B$ with $(\a_i,\b)\neq 0$.  Then $r_je_iB=r_jr_\b
r_iB$.  Also $(\a_i,r_j\b)\neq 0$ and so $e_ir_jB=r_{r_j\b}r_ir_jB$.  Now
again $r_{r_j\b}=r_jr_\b r_j$ giving the last term $r_jr_\b r_jr_ir_jB=r_jr_\b
r_iB$ as $r_i$ and $r_j$ commute and are of order two.

For (ii) we need to show $e_i r_j\xi_B =
r_je_i\xi_B$.  If $\a_i\in B^\perp$, then $\a_i$ is also in $(r_jB)^\perp$ and
so the result is $0$ in both cases.  If $\a_i \in B$, then $r_j\a_i=\a_i$ and
so $\a_i\in r_jB$.  Now $e_ir_j\xi_B=e_i\xi_{r_j B}h_{B,j}=\delta \xi_{r_j
B}h_{B,j} =r_j\xi_B \delta =r_je_i\xi_B$.  Suppose there is $\b\in B$ with
$\a_i\sim \b$.  Then also $\a_i\sim r_j\b$ and so $r_je_i\xi_B=r_jr_\b
r_i\xi_B =r_{r_j\b}r_ir_j\xi_B =e_i\xi_{r_jB}h_{B,j} = e_ir_j\xi_B$.

\nl (HCee).  Here $i\not\sim j$. We need to show $e_i e_jB = e_je_iB$ and $e_i
e_j\xi_B = e_je_i\xi_B$.  Suppose $\a_i\in B$. Then $e_je_iB = e_jB$.  Note
$\a_i\in e_jB$ in all cases and so also $e_ie_jB=e_jB$ so they are the same.
For the linear representation, $e_je_i\xi_B = e_j\xi_B\delta$.  If
$e_j\xi_B=0$ we are done as both sides of the required equality are $0$.
Otherwise, $e_j\xi_B = \xi_{B'}h$ with $h\in W(C)$ and $\a_i\in B'\in \B$ and
so also $e_ie_j\xi_B=e_j\xi_B\delta $, as required.  By symmetry of the
argument in $i$ and $j$, we may, and will, assume from now on that
$\a_i,\a_j\not\in B$.

Suppose next $\a_i\perp B$.  Then $e_iB=\cl{(B\cup \{\a_i\})}$.  We will use
the observation that, for $X\in\AO$ we have $e_kX = \cl{(X\cup\{\a_k\})}$
whenever $\a_k\in X\cup X^\perp$.  Suppose first $\a_j\perp B$.  Then, by
Lemma~\ref{lm:admclosure}(iii), $\a_j\in e_iB\cup {(e_iB)}^\perp$, so $e_je_i
B = e_j \cl{(B\cup \{\a_i\})}= \cl{(\cl{(B\cup \{\a_i\})}\cup\{\a_j\})}=
\cl{(B\cup \{\a_i,\a_j\})}$ is symmetric in $i$ and $j$, so we are done.  

If $\a_j\not\perp B$, then there is $b\in B$ with $\b\sim \a_j$. As
$\b\perp\a_i\perp\a_j$, using Lemma~\ref{lm:admclosure}~(iv), we find $e_je_i
B = r_\b r_j \cl{(B\cup \{\a_i\})}= \cl{(r_\b r_jB\cup \{\a_i\})}= e_i r_\b
r_jB = e_ie_jB$, as required.  For the representation, the arguments above for
all cases where $\a_i\perp B$ give $0$ here and there is nothing to prove. By
symmetry, we can suppose, for the remainder of the proof of (HCee), that
neither $\a_i$ nor $\a_j$ is in $B\cup B^\perp$.

This means there are $\b$ and $\b'$ in $B$ with $\a_i\sim \b$ and $\a_j\sim
\b'$.  Suppose $\a_j\not \sim \b$ and $\a_i\not \sim \b'$ and $\b\neq \b'$.
Then $e_ie_jB=e_ir_{\b'} r_jB=r_{\b}r_ir_{\b'} r_jB$ as we may use $\b$ to
give the action of $r_i$ (here we use that $r_{\b'}r_j\b=\b$).  Similarly
$e_je_iB=r_{\b'}r_jr_{\b} r_iB$.  By orthogonality of the roots involved in
commutation, $r_{\b}r_ir_{\b'} r_j=r_{\b'}r_jr_{\b}r_i$ and so $e_ie_jB
=r_{\b}r_ir_{\b'} r_jB = r_{\b'}r_jr_{\b}r_iB = e_je_iB$, as required.  For
the representation replace each $B$ by $\xi_B$ and the result follows.

We are done if such a choice of $\b$ and $\b'$ is possible.  Assume for the
remainder of the proof of (HCee) that such a choice is not possible.
During these arguments it will be useful to have a term for this.
We say $i$ and $j$ satisfy condition (*) if
\begin{itemize}
\item[]
$i\not\sim j$, there is a
$\b\in B$ with $\a_i\sim \b\sim \a_j$, and $B$ has
no pairs $\c$, $\c'$ for which
$\a_i\sim \c$, $\a_j\sim \c'$, $\a_i\not \sim \c'$, and $\a_j\not \sim \c$.
\end{itemize}

We suppose from now on in proving (HCee) that $i$ and $j$ satisfy condition
(*).  Suppose $\b$ is the only element of $B$ joined to $\a_i$ or $\a_j$.
Then $e_iB=\{\a_i\}\cup B\setminus\{ \b\}$ and $e_je_iB=\cl{(\{\a_j,\a_i\}\cup
B\setminus \{\b\})}$.  This is symmetric in $i$ and $j$ and we are done for
the poset part.  For the representation, as above, $e_i\xi_B=r_\b
r_i\xi_B=\xi_{\{\a_i\}\cup B\setminus\{\b\}}=\xi_{e_iB}$.  Now $\a_i$ is
orthogonal to all elements in $e_iB$ and so $e_je_i\xi_B=0$.  This is
symmetric in $i$ and $j$ and we are done.

The most difficult condition is when there is a second root $\c$ in $B$
also joined to both $\a_i$ and $\a_j$.  We assume for the moment there is no
such $\c$.  This means up to interchanging $i$ and $j$ there are $\c$ in $B$
with $\c\sim\a_i$ and $\a_j$ is not joined to any of them.  In fact now as $i$
and $j$ satisy (*), $\a_j$ is joined only to $\b$.  Now $e_jB=\{\a_j\}\cup
B\setminus \{\b\}$ and $e_ie_jB=r_{\c}r_i(\{\a_j\}\cup B\setminus \{\b\})$.
Notice $\a_j$ is in this set, so
$e_ie_jB = r_\c r_i r_\b r_j B$.  
Now consider $e_iB=r_{\c}r_iB.$ The only
element of $r_{\c}r_iB$ not perpendicular to $\a_j$ is $r_\c r_i \b$ and so
$e_je_iB =  (r_\c r_i r_\b r_i r_\c) r_j r_{\c}r_iB$.
But
\begin{eqnarray*}
 (r_\c r_i r_\b r_i r_\c) r_j r_{\c}r_i  &=& 
r_\c r_i r_\b r_i  r_j r_i  = r_\c r_i r_\b r_j,
\end{eqnarray*}
whence $e_ie_jB = e_je_iB$.  The same computations work for $\rho_\B$.

Suppose now that $B$ has roots $\b$ and $\c$, both joined to $\a_i$ and to
$\a_j$.  There are two cases to be considered. Since the roots in elements of
$\AO$ are all supposed to be positive, we will take the liberty of indicating
the positive root by its negative whenever convenient. Since confusion is
minimal, we shall write $\{\a\}$ rather than $\{\pm\a\}\cap\Phi^+$. By
changing positive roots to negatives we can assume that the inner products of
$\a_j$ with $\b$ and $\c$ are negative and that the inner product of $\a_i$
with $\c$ is negative.  There are now two choices $\pm 1$ for $(\a_i,\b)$.  

If $(\a_i,\b) = -1$, the Gram matrix has determinant~$0$ and an easy check
shows $-\a_i=\a_j+\b+\c$.  In this case the roots involved generate a root
system of type $\A_3$; an example of the configuration occurs for $\a_i =
\a_1$, $\c = \a_2$, $\a_j = \a_3$ and $\b = -(\a_1+\a_2+\a_3)$ with
$\a_1,\a_2,\a_3$ the simple roots of $\A_3$.

If $(\a_i,\b) = 1$, the roots involved generate a root system of type $\D_4$.
An example of the configuration occurs for $\a_i = \a_1$, $\c = \a_2$, $\a_j =
\a_3$ and $\b = \a_2+\a_3+\a_4$ where $\a_1,\a_2,\a_3,\a_4$ are the simple
roots of $\D_4$ with $2$ the triple node.

We suppose first that all roots $\b'$ of $B$ other than $\b$ and $\c$ are
orthogonal to $\a_i$ and $\a_j$.  In the $\A_3$ case we have $\c+\a_j+\b
=-\a_i$ and so $e_jB = r_{\b}(\{\b+\a_j , \c+\a_j\}\cup
B\setminus\{\b,\c\})= \{\a_j , \a_i\}\cup B\setminus\{\b,\c\}$, whence
$e_ie_jB = \{\a_j , \a_i\}\cup (B\setminus\{\b,\c\})$.  The other order gives
the same result and so $e_ie_jB=e_je_iB$.  For the representation,
$e_j\xi_B=r_\b r_j\xi_B$. We have just seen $e_jB =\{\a_j , \a_i\}\cup
(B\setminus\{\b,\c\})$.  This means $e_j\xi_B=\xi_{r_\b r_jB}$ and so
$e_ie_j\xi_B=r_\b r_j\xi_B\delta$.  Similarly, $e_je_iB = r_\b
r_i\xi_B\delta$.  For these to be equal we would need $r_i\xi_B=r_j\xi_B$.
{From} the definition this is $\xi_{r_iB}=\xi_{r_jB}$ and so is equivalent to
$r_iB = r_jB$. As, up to the signs of roots, $r_i\{\b,\c\} =
\{\b+\a_i,\c+\a_i\} = \{\b+\a_j,\c+\a_j\} =r_j\{\b,\c\}$, this is indeed the
case.

There is one other case in which all roots of $B$ other than $\b$ and $\c$ are
orthogonal to $\a_i$ and to $\a_j$, viz., $(\a_i,\b)=1$ and $\a_i,\a_j,\b,\c$
generate a root system of type $\D_4$.  Now $e_jB =r_\b r_j B=\{\a_j,
\c+\a_j+\b\}\cup (B\setminus\{\b,\c\})$.  Notice $(\a_i, \c+\a_j+\b)=0$ and so
$e_i(\{\a_j,\c+\a_j+\b\}\cup B\setminus\{\b,\c\})=
\cl{(\{\a_i,\a_j,\c+\a_j+\b\}\cup
B\setminus\{\b,\c\})}$.  Also $e_je_iB=e_j(\{\a_i, \b-\a_i-\c\}\cup
(B\setminus\{\b,\c\}))=(\{\a_j, \a_i, \b-\a_i-\c\}\cup
(B\setminus\{\b,\c\}))^{cl}$.  Now $(\c,\b-\a_i-\c)=-1$ and so in the closure
of $\{\a_j, \a_i, \b-\a_i-\c\}\cup (B\setminus\{\b,\c\})$ there is
$\a_i+\a_j+\b-\c-\a_i+2\c=\a_j+\b+\c$.  This means $\cl{(\{\a_j , \a_i,
\b-\a_i-\c\}\cup(B\setminus\{\b,\c\}))} = \cl{(\{\a_i , \a_j, \c+\a_j+\b,
\b-\c-\a_i\}\cup(B\setminus\{\b,\c\}))}$ and so $e_je_iB=e_ie_jB$.  For the
representation, the actions are all the $0$ action and so the required
equality is trivially satisfied.

This concludes the cases where $\a_i$ and $\a_j$ are joined only to $\b$ and
$\c$.  In the remaining cases, we may assume $\a_j$ is not orthogonal to at
least three roots in $B$ and so, because $B$ is admissible, $\a_j$ is
orthogonal to four roots of $B$.  This means there is $\eps\in B$ with
$(\a_j,\eps)=-1$ and $\eta=\b+\c+\eps+2\a_j$ is also in $B$.  
If $\a_i$ were not joined to
all the roots $\{\b,\c,\e,\eta\}$ but joined to another we would contradict
condition (*).  If it were joined to three it would be joined to four by the
admissibility.

If $\a_i$ were joined to all four of them consider the 4-dimensional linear
subspace of $\R^n$ spanned by the roots $\b, \c,\e,\eta$.  Both $\a_i$ and
$\a_j$ lie in this space and so the six roots generate a root system of type
$\D_4$. Here these elements must be two orthogonal vectors together with their
mates (meaning the unique other root $\zeta$ with the same set of orthogonal
roots in the 4-dimensional subspace). Also $\a_i$ and $\a_j$ must also be
orthogonal mates.  In this case the actions of $r_i$ and $r_j$ must be the
same and so the actions of $e_i$ and $e_j$ must be the same.  In particular
$e_ie_jB=e_i^2B=e_j^2B=e_je_iB$.  Also for the representation the actions of
$r_i$ and $r_j$ must be the same and $e_ie_j\xi_B=e_i^2\xi_B=e_je_i\xi_B$.

\nl The only remaining case occurs when $\a_i$ is joined to just $\b$ and $\c$
as discussed.  Now $\a_i$, $\a_j$, $\b$, $\c$, $\e$, and $\eta $ generate a
root system of type $\D_5$. Again $\b$, $\c$, $\e$, and $\eta $ 
are two pairs of roots together with their
orthogonal mates.  In computing $e_je_iB$ we can use $\b$ first and then
$\eta$ to get $e_je_iB=r_\eta r_jr_\b r_iB$.  In the other order we can use
$\eta$ first and then $\b+\a_j+\eta$ to compute
$e_ie_jB = r_{\b+\a_j+\eta}r_i r_\eta r_j
B$.  Now we use $r_{\b+\a_j+\eta}=r_\b r_\eta r_j r_\b r_\eta$ and derive
$$
r_{\b+\a_j+\eta}r_i r_\eta r_j = r_\b r_\eta r_j r_\b r_\eta r_i r_\eta r_j
=r_\eta r_\b  r_j r_\b    r_i r_j                                              
= r_\eta r_\b r_j r_\b r_jr_i =
 r_\eta r_j r_\b r_i .
$$ Thus, $e_je_iB = r_\eta r_j r_\b r_i B=r_{\b+\a_j+\eta}r_i r_\eta
r_jB=e_ie_jB$, as required.  The same computations work for $\rho_\B$, which
finishes the proof of (HCee).

\nl (HNrer).  Here $i\sim j$. We need to show $r_i e_jr_i B = r_je_ir_jB$
and $r_i e_jr_i \xi_B = r_je_ir_j\xi_B$.
Suppose first that there is $\b\in B$ with $r_i\b\not\perp\a_j$.
Then $r_ie_jr_iB = r_i r_{r_i\b} r_jr_i B = r_\b r_jr_ir_jB$
and  $r_ie_jr_i\xi_B = r_i r_{r_i\b} r_jr_i \xi_B
= r_\b r_jr_ir_j\xi_B$.
On the other hand, also $(\a_i,r_j\b) = (r_j\a_i,\b) = (r_i\a_j,\b) = (\a_j,r_i\b) \ne 0$,
so $\b\not\perp\a_i$ and so
$r_je_ir_jB = r_j r_{j\b} r_ir_j B = r_\b r_ir_jr_iB$
and $r_je_ir_j\xi_B = r_j r_{j\b} r_ir_j\xi_B = r_\b r_ir_jr_i\xi_B$. 
In view of the braid relation $r_jr_ir_j=r_ir_jr_i$,
both sides are equal.

Next suppose that $\a_j\perp r_iB$. Then $r_ie_jr_i B =
r_i(r_i B\cup \{\a_j\})^{cl} = ( B\cup r_i\{\a_j\})^{cl}$
and $e_jr_i\xi_B =e_j\xi_{r_iB}$ or $0$ if
$\a_i\perp B$.
Also, $\a_i = r_jr_i\a_j\perp r_jr_ir_iB = r_jB$ and so
$r_je_ir_j B = ( B\cup r_j\{\a_i\})^{cl}$ and $e_ir_j \xi_B =0$. 
As $r_i\a_j = r_j\a_i$, the two sides are equal.

Finally, suppose that $\a_j\in r_iB$.  This means $\a_i+\a_j\in B$. Now $r_i
e_jr_i B = r_i r_i B = B$. Moreover, $\a_i= r_jr_i\a_j\in r_j B$ and so $r_j
e_ir_j B = r_j r_j B = B$, as required.  For the representation, this means
$r_ie_jr_i\xi_B=r_jr_ir_{r_i\a_j}r_i\xi_B$. However $r_jr_ir_jr_ir_jr_i=1$
and so both sides are the same.

\nl (HNrre).  Here $i\sim j$. For (i), we need to show $r_jr_ie_jB = e_ie_jB$.
As $\a_j$ is in $e_j(B\setminus\a_i^\perp)$, we have $e_ie_jB=r_j
r_ie_jB$, and we are done.

For (ii) we need to show $r_jr_ie_j\xi_B = e_ie_j\xi_B$.  We may assume
that $e_j\xi_B$ is not $0$.  If $\a_j\in B$, then $e_ie_j\xi_B=\delta
e_i\xi_B=\delta r_j r_i\xi_B$.  As $e_j\xi_B=\delta \xi_B$ we are done.  If
there is $\b \in B$ not perpendicular to $\a_j$, then $e_j\xi_B=r_\b
r_j\xi_B=\xi_{r_\b r_iB}h$. Notice $r_\b r_jB$ contains $\a_j$.  Now
$e_ie_j\xi_B=r_j r_ie_j\xi_B$, and we are done.

\np There is one more property we need to show: if $ Y\subseteq X$, then
$aY\subseteq aX$.  The action by $r_i$ is just the group action which preserves
inclusion so we need only check the actions by $e_i$.
Let $Y\subseteq X$.

Suppose $\a_i\in X$.  Then $e_iX=X$.  If $\a_i\in Y$ then $e_iY=Y$ and we are
done.  If $\a_i\not \in Y$ then $\a_i \perp Y$ as $\a_i\in X$ and elements in
$X$ are mutually orthogonal.  Consequently, $e_iY=(\{\a_i\} \cup Y)^{cl}
\subseteq X = e_iX$, as required.

For the remainder of the proof, we may assume $\a_i \not \in X$.  If
$\a_i\perp X$ then $\a_i\perp Y$.  This means $e_iY=(\{\a_i\}\cup Y)^{cl}$ and
$e_iX=(\{\a_i\}\cup X)^{cl}$ so $Y\cup \{\a_i\}\subseteq X\cup \{\a_i\}$, and
hence $e_iY\subseteq e_iX$, as required.

Suppose that there is $\b\in Y$ with $\a_i\sim\b$.  Then $e_iY=r_\b r_iY$
and $e_iX=r_\b r_iX$ while $r_\b r_i Y \subseteq r_\b r_i X$.  The only case
left is $\a_i \perp Y$ but there is $\b$ in $X$ with $(\a_i,\b)\neq 0$.
Clearly $\b\not \in Y$.  Now $e_iY=(\{\a_i\}\cup Y)^{cl}$ and
$e_iX=r_\b r_iX$.   By Lemma~\ref{lm:admclosure}(ii),
$\a_i\in e_iX$.
As $r_\b r_iY = Y$, we find $e_iY = (\{\a_i\}\cup Y)^{cl} = 
(\{\a_i\}\cup r_\b r_i Y)^{cl} \subseteq \cl{(e_i X)} = e_i X$,
so the assertion holds.
\end{proof}

\begin{Cor}\label{cor:semi-direct}
For $X$ the highest element of $\B$, the permutation stabilizer $N_W(X)$ of
$X$ in $W$ is the semi-direct product of $K_X$ and $W(C_\B)$.
\end{Cor}

\begin{proof}
{From} (ii) of the theorem we see
$N_W(X)=\{w\in W\mid w\xi_X \in\xi_X W(C_\B)\}$.
As $h_{X,i} = r_i$ for $i$ a node of $C_\B$, the subgroup $W(C_\B)$ of
$N_W(X)$
satisfies
$W(C_\B)\xi_X = \xi_X W(C_\B)$ and so is a complement to $K_X$ in $N_W(X)$.
\end{proof}

\section{Rewriting elements and upper bounding the dimension}
\label{sec:upb}
The main goal of this section is to prove that every element of $\BrM(M)$ can
be written in a certain standard form, which corresponds to the well-known
Brauer diagrams if $M = \A_{n-1}$. This will lead to the following upper bound
of the dimension of $\Br(M)$.  Recall that $C_\B$ is the set of nodes of $M$
whose corresponding roots are orthogonal to the highest element of $\B$.

\begin{Prop}\label{prop:upb}
The dimension of the Brauer algebra of type $M$ is at most
$$ \sum_{\B} |\B|^2 |W(C_\B)|.$$
\end{Prop}

This will be proved in a series of lemmas and propositions
and completed at the end of this section.

\begin{lemma}\label{lemma:norm_e}
Let $i$ and $j$ be nodes of $M$.
If $w\in W$ satisfies $w\a_i = \a_j$, then
$w e_i w^{-1} =  e_j$.
\end{lemma}

\begin{proof}
By \cite[Proposition 3.2]{CGW}, there is a unique element $w_{ij}$ of minimal
length such that $w_{ij}\a_i=\a_j$.  This can be proved exactly as in
\cite[Lemma 3.1(iv)]{CGW}, by use of (HNree) and (HNeer).  It remains to
verify that $C_W(\a_i)$ centralizes $e_i$.  This is proved as in \cite[Lemma
3.9]{CGW}, where it was shown $s_is_\b = s_\b s_i$ for any root $\b$ of $W$
orthogonal or equal to $\a_i$, where $s_\b$ is the product in the Artin group
of the simple generators corresponding to a minimal length word for $r_\b\in
W$.  Here we replace $s_i$ by $e_i$ and use (HNree) and (HNeer) appropriately.
Since $C_W(\a_i)$ is generated by such reflections $r_\b$, this establishes
the lemma.
\end{proof}

Consider a positive root $\b$ and a node $i$ of $M$.
There exists $w \in W$ such that $\b = w\a_i$. 
Define the element $e_\b$ of $\Br(M)$ by
\begin{equation}\label{def:e_t}
e_\b = w e_i w^{-1}.
\end{equation}

Lemma \ref{lemma:norm_e} 
implies that $e_\b$ is well defined.
The relations in $\Br(M)$ involving the elements $e_\b$
extend the relations already described for
fundamental elements $e_i$. 

\begin{lemma}\label{lemma:relations}
Let $\b$ and $\c$ be positive roots of $W$.
\begin{itemize}
\item[(i)] $e_\b  r_\b = r_\b e_\b  = e_\b $ and $e_\b^2=\delta e_\b$.
\item[(ii)] If $(\b,\c) = \pm1$ then 
\begin{enumerate}[(a)]
\item $e_\b  r_\c  e_\b  = e_\b $,
\item $r_\b r_\c  e_\b  = e_\c  r_\b r_\c  = e_\c e_\b $,
\item $e_\b  e_\c e_\b  = e_\b $.
\end{enumerate}
\item[(iii)] If $(\b,\c) = 0$, then $e_\b  r_\c = r_\c  e_\b $ and $e_\b e_\c  =e  _\c e_\b.$
\end{itemize}
\end{lemma}

\begin{proof}
If $\b$ and $\c$ are simple roots, this is direct from the defining
relations of $\BrM(M)$. Otherwise, there are $w\in W$ and nodes $i,j$ of $M$
such that
$w\a_i = \b$ and $w\a_j = \c$, and
the result follows from (\ref{def:e_t}) by conjugation.
\end{proof}

We next extend the definition of $e_\b$ to arbitrary sets of mutually orthogonal
positive roots. For such a set $B$, we define the element $e_B$ of $\Br(M)$ by
\begin{equation}\label{def:e_X}
e_B = \prod_{\b \in B} e_\b.
\end{equation}

This definition is unambiguous as $e_\b$ and $e_\c$ commute whenever $\b$ and
$\c$ are orthogonal (cf.~Lemma \ref{lemma:relations}(iii)). Clearly, $e_B$
behaves well under conjugation by $W$ in the sense that $ue_Bu^{-1} = e_{uB}$.

An important difference between $\Br(\A_n)$ and the Brauer algebras of other
types is the fact that the orbit of $B$ under the action of $W$ need not
correspond bijectively with the orbit of $e_B$ under $W$ by conjugation.  For
example, when $M = \D_4$, with the labeling of the nodes as in \cite{Bou}, the
set $B = \{\a_1,\a_3,\a_4\}$ is distinct from $wB$, where $w = r_2r_1r_3r_2$,
but $we_Bw^{-1} = e_B$.  For this reason, we need compare the action of $W$ on
$e_B$ with the conjugation action on its admissible closure $\cl{B}$ rather
than $B$.  The necessary transition from $B$ to $\cl{B}$ is expressed in the
next lemma.

\begin{lemma}\label{lemma:adm_idmptnts}
If $X$ is a set of mutually orthogonal positive roots of $W$, then
\[
e_{\cl{X}} =  e_X \delta^{|\cl{X}|-|X|}.
\]
\end{lemma}

\begin{proof}
Suppose that $\a \in \Phi^+\setminus X$ is non-orthogonal to
the pair
$\b_2$, $\b_3$ in $X$.  
By Lemma \ref{lemma:relations}, 
\begin{eqnarray*}
r_\a r_{\b_2} r_{\b_3} r_\a e_{{\b_2}} e_{{\b_3}} 
&=&   r_\a r_{\b_2} r_{\b_3} r_\a e_{{\b_3}} e_{{\b_2}} 
 =  r_\a r_{\b_2} e_\a e_{{\b_3}} e_{{\b_2}}
 =  e_{\b_2} e_\a e_{{\b_3}} e_{{\b_2}}  \\
&=  & e_{\b_2} e_\a e_{{\b_2}} e_{{\b_3}}  
 =  e_{{\b_2}} e_{{\b_3}},
\end{eqnarray*}
whence $r_\a r_{\b_2} r_{\b_3} r_\a e_X = e_X$.

Suppose now that $\b_1$ is a third root of $X$ that is not orthogonal to $\a$.
Let $\c$ be the unique positive root in $X$ non-orthogonal to $\a$ and
orthogonal to $\b_1$, $\b_2$, and $\b_3$, cf.~Lemma \ref{lemma:unique_root}.
Then $\c = r_\a r_{\b_2} r_{\b_3} r_\a \b_1$.  As $r_\a r_{\b_2} r_{\b_3} r_\a
r_{\b_1} =r_\c r_\a r_{\b_2} r_{\b_3} r_\a $,
using Lemma \ref{lemma:norm_e}  we find
$$
r_\a r_{\b_2} r_{\b_3} r_\a e_X \delta 
= r_\a r_{\b_2} r_{\b_3} r_\a e_{\b_1}e_X  
 =  e_\c r_\a r_{\b_2} r_{\b_3} r_\a e_X  
 =   e_\c   e_X 
 =  e_{\{\c\}\cup X} .
$$
This procedure can be repeated until we have reached $\cl{X}$.
The lemma follows.
\end{proof}

\begin{Prop}\label{prop:rep}
Let $X $ be an admissible set and let
$Y$ be a set
of mutually orthogonal positive roots (not necessarily admissible). Then
\[
e_Y \xi_X  \in  \xi_Z W(C) \delta^k \cup \{0\},
\]
for some $k \in \mathbb{N}$ with $k \le |Y|$ and $Z \in WX$ with $Y
\subseteq Z$.  Moreover, $e_Y\xi_X\ne0$ with $k = |Y|$ if and only if $Y
\subseteq X$, in which case $e_Y\xi_X= \xi_X \delta^{|Y|}$.
\end{Prop}

\begin{proof}
The proof of the first assertion is by induction on the size of $Y$.

Suppose that $|Y| = 1$. There exists a positive root $\a$ such that $Y = \{ \a
\}$ and $e_Y = e_\a$.  If $\a$ is not a simple root, choose $w\in W$ for
which $w \a_i =\a $ where $\a_i$ is simple. Then $e_\a=we_iw^{-1}$ and
$w\xi_X=\xi_{wX}h$ with $h\in W(C)$.  The conditions on subsets are preserved.
Therefore, we may, and shall, assume that $\a$ is simple.
There are three cases to consider.

\nl $\a \in X$. Then $e_Y \xi_X = e_\a \xi_X = \xi_X \delta$. Now, for
$k = 1$ and $Z = X$ we have $Y \subseteq Z$ and $k = |Y|$, as required.

\nl $\a\perp X$. Then $e_Y \xi_X = e_\a \xi_X  = 0$ and the assertions
hold.

\nl $\a\sim\b\in X$. Then
$e_Y \xi_X 
 =  e_\a \xi_X  
 =  r_\b r_\a \xi_{X }
\in   \xi_{ r_\b r_\a  X }  W(C)$.
Moreover, $\a = r_\b r_\a\b \in r_\b r_\a X $, so the assertions hold with 
$Z = r_\b r_\a X $ and  $k = 0$.

Next, assume $|Y| > 1$.
Take $\a \in Y$ and set $Y_0 = Y \backslash \{ \a \}$. We compute
$e_Y \xi_X =  e_\a e_{Y_0}  \xi_X$.
If $e_{Y_0} \xi_X = 0$, then clearly
$e_Y \xi_X = 0$.
Assume therefore $e_{Y_0} \xi_X \ne 0$.
By the induction hypothesis,
$ e_{Y_0} \xi_X =  \xi_Z v \delta^k$ 
with $Y_0 \subseteq Z$, $v
\in W(C)$, and $k \leq |Y_0|$. Now
\begin{equation}\label{eqn:induction2}
e_Y \xi_X  =  e_\a \xi_Z v \delta^k .
\end{equation}
Put $Z_0 = Z \setminus Y_0 $. 
We have $\a\in Y_0^\perp$. Moreover, every element of $Z_0$ commutes
with every element of $Y_0$ . Again,
there are three cases to consider,

\nl $\a \in Z_0$. Then $e_Y \xi_X = e_\a \xi_Z v \delta^k = \xi_Z v \delta^{k
+ 1}$ and $Y = Y_0 \cup \{ \a \} \subseteq Z$. Furthermore, $k \leq |Y_0| =
|Y| - 1$, so $k + 1 \leq |Y|$. This proves the proposition in this case.

\nl $\a \perp Z$. 
Then $e_\a \xi_Z  = 0$, so $e_Y \xi_X =
 e_\a \xi_Z v \delta^k = 0$ by (\ref{eqn:induction2}).

\nl $\a\sim \b \in Z$.
Then
\begin{eqnarray*}
e_Y \xi_X 
& = &  e_\a \xi_Z v \delta^k 
 =   r_\b r_\a  \xi_{Z} v \delta^k.
\end{eqnarray*}
Now $r_\b r_\a Z = r_\b r_\a Z_0 \cup r_\b r_\a Y_0 $. As $\a,\b\perp Y_0$, we
haved $r_\b r_\a Y_0 = Y_0$. Hence $r_\b r_\a Z = r_\b r_\a Z_0 \cup Y_0$. As
before, $\a \in r_\b r_\a Z_0 $. Hence $Y = Y_0 \cup \{ \a \} \subseteq r_\b
r_\a Z$.  Furthermore, $k \leq |Y_0| < |Y|$, as required for the proof of the
first assertion.

In order to settle the second assertion, suppose that $k = |Y|$ and
$e_Y\xi_X\ne0$. If $Y=\emptyset$ the assertions $Y\subseteq X$ and $e_Y\xi_X=
\xi_X \delta^{|Y|}$ hold trivially.  Let $k>0$ and proceed by induction on
$k$.  Take $\b\in Y$ and set $Y'=Y\setminus\{\b\}$.  Clearly $e_{Y'}\xi_X\ne0$
and $k-1 = |Y'|$, so, by the induction hypothesis, $Y'\subseteq X$ and
$e_{Y'}\xi_X= \xi_X \delta^{k-1}$, whence $e_Y\xi_X = \e_\b\xi_X\delta^{k-1}$.
If $\b\perp X$, then $e_Y\xi_X = 0$, a contradiction.  If $\b\sim\c\in X$,
then $e_Y\xi_X = r_\c r_\b \xi_X \delta^{k-1}\in \xi_X W(C) \delta^{k-1}$
contradicting the assumption $e_Y\xi_X \in \xi_X W(C) \delta^k$, so we
must have $\b\in X$.  It follows that $Y = Y'\cup\{\b\}\subseteq X$ and
$e_Y\xi_X = \xi_X\delta^k$ as required for the only if part.  For the converse
use the case $|Y|=1$ above repeatedly.  This establishes the second assertion.
\end{proof}

\begin{Cor}\label{thm:rep}
Let $\B $ be an admissible $W$-orbit and $X, Y \in
\B$. 
Then
\[
e_Y \xi_X \in  \xi_Y W(C) \delta^k \cup \{ 0\},
\]
where $k \leq |Y|$. Moreover, if $k = |Y|$ and $e_Y\xi_X\ne0$, then $Y = X$
and $e_X\xi_X=\xi_X\delta^k$.
\end{Cor}

\begin{proof}
Suppose that $e_Y\xi_X \neq 0$. By Proposition \ref{prop:rep} there are $Z \in
W\, X$, $w \in W(C)$, and $k \in \mathbb{N}$ such that $e_Y\xi_X = w \xi_Z
\delta^k$.  Moreover, $Y \subseteq Z$ and $k \leq |Y|$. Since $Z \in \B$ we
know that $|Y| = |X| = |Z|$. Thus $Y = Z$.

Suppose that $k = |Y|$. Then $Y \subseteq X$ by Proposition
\ref{prop:rep}. Since $|Y| = |X|$ we conclude $Y = X$.
\end{proof}

For $X$ a set of mutually orthogonal positive roots, define
the {\em annihilator} of $e_X$, denoted $A_X$, to be
\begin{equation}
A_X = \{ w \in W \mid w e_X = e_X \}.
\end{equation}
and the {\em centralizer} of $e_X$, denoted $N_X$, to be
\begin{equation}
N_X = \{ w \in W \mid e_X w = w e_X \}.
\end{equation}
In view of Lemmas \ref{lm:admclosure}(iv) and \ref{lemma:adm_idmptnts},
$N_W(X)\leq N_W(\cl{X})\leq N_X$.  Also, by Proposition \ref{prop:rep},
$A_X\leq A_{\cl{X}} \lhd N_X$. Some further properties of these subgroups are
listed in the next proposition, the second item of which we could only prove
by means of a case by case verification.

Before Lemma \ref{lm:Hgens} we introduced the notation $K_X$
for the kernel of the
restriction of $\rho_{\B}$ to $N_X$ on $\xi_X\Z[W(C),\delta^{\pm1}]$.

\begin{Prop}\label{prop:AX=Kernel}
Let $X$ be the highest element in its $W$-orbit and
put $C = C_{WX}$.
\begin{enumerate}[(i)]
\item $N_X = N_W(X)$.
\item The normal subgroup $A_X$ of $N_X$ coincides with $K_X$. It is
generated by $$\{r_\b,r_\a r_\b r_\c r_\a\mid
\a\in\Phi^+, \b,\c\in X , \b\sim\a\sim\c\}.$$
\item $N_X$ is the semi-direct product
of $A_X$ and $W(C)$. 
\end{enumerate}
\end{Prop}

\begin{proof}
(i). 
Above, we observed that $N_W(X) \leq N_X$.
By Proposition \ref{prop:rep}, 
$we_X = e_X w$ for $w\in W$ implies $\xi_{wX}\in \xi_X W(C)\delta^\Z$.
Therefore, $N_X$ leaves invariant the 1-dimensional subspace
$\xi_X\Z[W(C),\delta^{\pm1}]$ of $V_\B$. This proves $N_X\leq N_W(X)$.

\nl(ii). 
If $w\in W$ satisfies $we_X = e_X$,
then there is $h\in W(C)$ such that
$\xi_{wX}h\delta^\Z  = w\xi_X\delta^\Z = we_X\xi_X \delta^\Z
= e_X\xi_X\delta^\Z= \xi_X\delta^\Z$.
But then $wX = X$, so $w\in  N_W(X) = N_X$ by (i), and $h = 1$.
This proves that $A_X$ is contained in $K_X$.

Let $L_X$ be the subgroup of $W$ with the generators specified in
the assertion. If $\b\in X$ then, by Lemma \ref{lemma:relations}, 
$r_\b e_X = r_\b e_\b e_X\delta^{-1}
=  e_\b e_X\delta^{-1} =   e_X$, so $r_b\in A_X$.
Let $\a\in\Phi^+$
and assume $\b$ and $\c$ in $X$ are as stated.
Then $r_\a r_\b r_\c r_\a\in A_X$ by the first paragraph of 
the proof of Lemma \ref{lemma:adm_idmptnts}.
Hence $L_X$ is contained in $A_X$. Now $L_X$ is a normal subgroup
of $N_W(X)$ contained in $K_X$,
so the product $L_XW(C)$ is a subgroup of $N_X$.
A case by case analysis shows that the action of $L_X$ induced on $X$
coincides with the action of $N_X$. Also, by inspection of
cases, for every root
of $\b\in\Phi$ orthogonal to $X$
there is an element $u\in L_X$ with $ur_\b\in W(C)$.
This implies that $L_X$ coincides with $K_X$ and hence with $A_X$.

\nl(iii).  By (i) and (ii), this is a restatement of Corollary
\ref{cor:semi-direct}.
\end{proof}

\begin{lemma}\label{lemma:action}
Let $X$ be a set of mutually orthogonal positive roots, $w \in W$, and $\b \in
\Phi^+$.
\begin{itemize}
\item[(i)] If $X \in  \AO$ and $w\in W$ is of minimal length in its coset $wN_X$, then
$w \xi_X  =  \xi_{wX}$.
\item[(ii)] The product $e_\b e_X$ 
can be expanded as follows.
\[
e_\b e_X   = \left\{
\begin{array}{ll}
 e_X \delta & \mbox{if $\b \in \cl{X}$,} \\
e_{X \cup \{ \b \}} & \mbox{if $\b \perp X$,} \\
r_\c r_\b e_{ X } & \mbox{where $\c \in X$, if $\b \sim \c\in X $}.
\end{array}
\right.
\]
\end{itemize}
\end{lemma}

\begin{proof}
(i).
In a minimal expression $s_1\cdots s_q$ 
of $w$ as a product of simple reflections, each
$s_i$ will move $s_{i+1}\cdots s_q X$.
Then $s_i\xi_{s_{i+1}\cdots s_q X} = \xi_{s_{i}\cdots s_q X}$.

\nl(ii).  If $\b \in \cl{X}$, the result follows from Lemma
\ref{lemma:adm_idmptnts}.  If $\b \perp X$, the assertion is immediate from
the definition of $e_{\{\b\}\cup X}$.  Finally, suppose that there is some $\c
\in X$ with $\b\sim\c$. Then the assertion follows from Lemma
\ref{lemma:relations}(ii)(b).
\end{proof}

Let $\AO_0$ be the set of highest elements from the $W$-orbits in $\AO$.  For
$X\in \AO_0$, let $D_X$ be a set of right coset representatives for $N_X =
N_W(X)$ in $W$. 
By convention if $X=\emptyset$ we take $e_\emptyset$ to be the 
identity, $N_{W\emptyset}$ to be $W$, and $C_{W\emptyset}$ also to be $W$.  

\begin{Prop}\label{prop:monomials}
Each element of the Brauer monoid $\BrM(M)$ can be written in the form
$u e_X zv \delta^k $, where $X\in \AO_0$, 
$u,v^{-1} \in D_X$, $z \in W(C_{WX})$, and $k \in \Z$.
\end{Prop}

\begin{proof}
By Lemma \ref{lemma:action}(ii), any expression of the form $e_\b w e_{X'}$
with $\b\in\Phi^+$, $w\in W$ and $X'$ a set of mutually orthogonal positive
roots, can be rewritten in the form $ve_Y\delta ^k$ with $v\in W$, $Y\in\AO$
and $k\in \Z$.  Consequently, up to a power of $\delta$, every element of
$\BrM(M)$ is equal to $w_1e_Xw_2^{-1}$ for some $X\in\A_0$ and $w_1,w_2\in W$.
Now, using Proposition \ref{prop:AX=Kernel}(iii), write $w_1 = u y_1 z_1$ and
$w_2 = v y_2z_2$ with $u,v\in D_X$, $y_1,y_2\in A_X$, and $z_1,z_2\in W(C)$.
Then $w_1e_Xw_2^{-1} = uy_1e_Xz_1z_2^{-1}y_2^{-1} v^{-1} =
ue_Xz_1z_2^{-1}y_2^{-1} v^{-1} = uz_1z_2^{-1}e_Xy_2^{-1} v^{-1} =
uz_1z_2^{-1}e_X v^{-1} = ue_X z_1z_2^{-1} v^{-1}$.  Taking $z = z_1z_2^{-1}$,
we find an expression as required.
\end{proof}

\nl {\bf Proof of Proposition \ref{prop:upb}}.  The dimension of $\Br(M)$ is
equal to the size of the quotient monoid $\BrM(M)\la \delta^{\pm1}\ra$, which,
by Proposition \ref{prop:monomials}, is at most
$\sum_{X\in\AO_0}|D_X|^2\cdot|W(C_{WX})|$. 
The proposition follows as $|D_X| =|WX|$.

\begin{remark}\label{rmk:triple1}
\rm
To finish this section, we describe the usual Brauer diagram on $n$ strands
corresponding to $ue_Xzv\delta^k$ for $k\in\N$ when $M=\A_{n-1}$.  It contains
$k$ circles. The horizontal strands at the top are determined by $uX$ in the
following way: each root in $uX$ is of the form $\eps_i-\eps_j$ in the
standard representation of $\Phi^+$, where each $\eps_t$ denotes the $t$-th
standard basis vector of $\R^n$; in the diagram there is a corresponding
horizontal strand from $i$ to $j$.  The bottom of the diagram is obtained by
the same interpretation of $v^{-1}X$. Finally the element $z$ determines the
vertical strands in terms of a permutation on the remaining nodes up to a
translation from the highest root to $X$. See Remark \ref{rmk:triple2} below
on how to obtain it.
\end{remark}

\section{Irreducibility of representations and lower bounding the dimension}
\label{sec:lb}
Corollary~\ref{thm:rep} allows us to find irreducible representations of
the Brauer algebra.  In fact one for each pair of a $W$-orbit $\B$ inside
$\AO$ and an irreducible representation of $W(C_\B)$.  This will enable us to
find a lower bound for the dimension of the Brauer algebra, which together
with Proposition \ref{prop:upb} gives the exact dimension.  Fix a $W$-orbit
$\B$ inside $\AO$ and recall the notation $\rho_\B$ from Theorem
\ref{prop:BrMonAction}(ii). We shall often abbreviate $V_\B$ and $C_\B$ to $V$
and $C$, respectively, where $V_\B$ was defined just above Lemma~\ref{lm:rep}.

\begin{Prop}\label{prop:eBvnotzero}  
Suppose that $v = \sum \xi_Bw \lambda_{B,w}$ is a nonzero element of $V$ where
the sum is over all $w\in W(C)$ and over all $B\in\B$.  Then there is some
$Y\in \B$ for which $e_Yv\neq 0$.
\end{Prop}

\begin{proof}
Suppose that $e_Yv= 0$ for all $Y\in \B$.  By Proposition
\ref{prop:rep} there are coefficients
$T_{Y,u;B,w}\in\{0\}\cup \delta^\Z$, where $Y,B\in\B$
and $u,w\in W(C)$, such that 
$$e_Y\xi_B w = \sum_{u\in W(C)} \xi_Y u T_{Y,u;B,w}.$$
After an ordering of $\B\times W(C)$, the coefficients
$T_{Y,u;B,w}$ can be considered entries of a square matrix, $T$, over
$\Q(\delta)$ whose rows and columns are both indexed by the pairs
in $\B\times W(C)$.

Let $\lambda $ be the column vector with entries 
$\lambda_{B,w}$ indexed in the same order as used for $T$.
Now
\[
0 = e_Yv  =
\sum e_Y \xi_Bw \lambda_{B,w}
=
\sum_{u\in W(C)} \sum_{B,w} \xi_Y u T_{Y,u;B,w}   \lambda_{B,w}
= \sum_{u\in W(C)}  \xi_Y u (T\lambda)_{Y,u}.
\]
As this equality holds for all $(Y,u) \in\B\times W$, we find $T \lambda =
0$.  By Corollary~\ref{thm:rep}, the exponent of $\delta$ in an entry
$T_{Y,u;B,w}$ of $T$ is $|B|$ on the diagonal as $e_B\xi_Bw=\delta^{|B|}w$,
whereas, at nonzero off-diagonal entries, only lower powers of $\delta$ occur.
Consequently, $\det(T)$ is a nonzero element of $\Q[\delta^{\pm1}]$.  This
means that $T$ is nonsingular over the field $\Q(\delta)$, and so $T\lambda =
0$ implies $\lambda = 0$, that is, $v = 0$, a contradiction.  Hence the
proposition.
\end{proof}

\begin{Prop}\label{prop:invariantsubspace}  Suppose that $U$ is the regular 
representation of $W(C)$ over $\Q(\delta)$ and $U_1$ is an invariant subspace
of $U$ for $W(C)$.  Then $\sum_{B\in \B} \xi_BU_1$ is an invariant subspace of
$V\otimes_{\Z[\delta^{\pm1}]}
\Q(\delta)$ for $\Br(M)$.
\end{Prop}  

\begin{proof}  This follows from the actions of $r_i$ and $e_i$ on 
$\xi_Bu$ for $u\in U$.  In each case the result is of the form $0$ or
$\xi_Ywu$ with $w\in W(C_\B)$, and so if $u\in U_1$ then so is $wu_1$.
\end{proof}

\begin{Prop}\label{prop:irreducible}  
If $U_1$ is an irreducible invariant subspace of the $\Q(\delta)[W(C)]$-module
$U$, then the representation given
by Proposition \ref{prop:invariantsubspace} gives an irreducible
representation of $\Br(M)$ over $\Q(\delta)$.
\end{Prop}

\begin{proof}  
Let $v$ be a nonzero vector in $V_1=\sum_{B\in \B} \xi_B U_1$.  We know from
Proposition~\ref{prop:eBvnotzero} that there is a $B\in\B$ for which $e_Bv\neq
0$.  Suppose that $v$ is a nonzero element of an invariant subspace of $V_1$.
Then this subspace also contains $e_Bv$, which, by Corollary \ref{thm:rep},
is equal to $\xi_B u_1$ for some nonzero $u_1\in U_1$.  
As the representation is
irreducible, by acting by $W(C_\B)$ we can obtain all of $\xi_BU_1$ in the
invariant subspace.  As $W$ is transitive on $\xi_Y$ for $Y\in \B$,
the invariant subspace contains $\sum_{B\in\B} \xi_BU_1$ 
and so coincides with $V_1$. Therefore, the representation is irreducible.
\end{proof}

\begin{Prop}\label{prop:different}   
The irreducible representations obtained in 
Proposition~\ref{prop:irreducible} are not equivalent.  
\end{Prop}

\begin{proof}  
Suppose that $U_1$ and are $ U_2$ are inequivalent irreducibles of $W(C_\B)$.
Restrict to $W(C_\B)$.  The corresponding representations are $|\B|U_1$ and
$|\B|U_2$ which are inequivalent.  This means the representations of $\Br(M)$
cannot be equivalent.
\end{proof}

\nl {\bf Proof of Theorem \ref{th:main}.}  The above shows that, for each
irreducible representation $\tau$ of $W(C_\B)$ there is an irreducible
representation $\rho_{\B}\otimes\tau$ of $\Br(M)\otimes_{\Z[\delta^{\pm1}]}
\Q(\delta)[W(C_\B)]$.  In particular, the algebra
$\Br(M)\otimes_{\Z[\delta^{\pm1}]} \Q(\delta)$ maps homomorphically onto a
direct sum of matrix algebras of dimensions $|\B|\tau(1)$ over $\Q(\delta)$
for $\B$ running over the admissible $W$-orbits in $\AO$ and $\tau$ over the
irreducible representations of $W(C_\B)$.  Therefore, $\dim(\Br(M))\ge
\sum_{\B,\tau} |\B|^2\tau(1)^2 = \sum_{\B} |\B|^2 |W(C_\B)|$.  In Proposition
\ref{prop:upb}, this number was proved to be an upper bound for
$\dim(\Br(M))$, so, in view of Lemma \ref{lm:counting}, the homomorphism onto
a direct sum of matrix algebras is an isomorphism and
$\Br(M)\otimes_{\Z[\delta^{\pm1}]} \Q(\delta)$ is split semisimple, so Theorem
\ref{th:main} is proved.

\np
With the notation of Proposition \ref{prop:monomials} and as an immediate
consequence of this proposition and the theorem, we have the following
two corollaries.

\begin{Cor}\label{cor:basis}
For $M\in\ADE$, the Brauer algebra $\Br(M)$ over $\Z[\delta^{\pm1}]$ has
a basis of the form $ue_Xzv$ for $X\in\AO_0$, $u,v^{-1}\in D_X$, 
and $z\in W(C_{WX})$.
\end{Cor}

\begin{Cor}\label{cor:ss}
For $M\in\ADE$, the Brauer algebra
$\Br(M)\otimes_{\Z[\delta^{\pm1}]}\Q(\delta)$ over $\Q(\delta)$ is a direct
sum of matrix algebras of size $|\B|\cdot\tau(1)$ for $(\B,\tau)$ running over
all pairs of a $W$-orbit $\B$ inside $\AO$ and an irreducible representation
$\tau$ of $W(C_\B)$.
The irreducibles are indexed by the irreducibles of $W(C_\B)$ 
over all $\B$.  
\end{Cor}

\begin{remark}\label{rmk:triple2}
\rm We finish by describing how to compute from a Brauer monomial $a\in
\BrM(M)$ the triple $(L,R,z)$ consisting of two elements $L$, $R$ of the same
$W$-orbit $\B = WX$ inside $\AO$, where $X\in\AO_0$, and of the element $z\in
W(C_\B)$ for which $a= ue_Xzv\delta^k$ as in Proposition \ref{prop:monomials}
with $L = uX$ and $Y = vX$. First compute $L = a\emptyset$ and $R =
a^{\op}\emptyset$, where $a^\op$ is the element of $\BrM(M)$ obtained by
reading backwards an expression of $a$ as a word in the generators (this
element is well defined as the operation ${\cdot}^\op$ is an anti-involution,
see \cite{CGW} or note that the set of relations shown in Tables \ref{table:BrauerRels} and
\ref{AddBrauerTable} is invariant under opposition). As a consequence of Proposition \ref{prop:monomials}, $L$ and
$R$ belong to the same $W$-orbit inside $\AO$. Let $X\in\AO_0$ be the highest
element of this orbit.  Pick $u,v^{-1} \in D_X$ such that $L = uX$ and $R =
v^{-1} X$.  Now compute $u^{-1}av^{-1}\xi_X$. The result will be an element of
the form $\xi_X z\delta^s$ for some $s\in\Z$ and $z\in W(C_{WX})$.  Then $a =
u e_X z v \delta^k$ with $k = s-|X|$, as required.  As discussed in Remark
\ref{rmk:triple1}, for $M=\A_{n-1}$, the sets $L$ and $R$ determine the
horizontal strands at the top and bottom, respectively, of the corresponding
Brauer diagram, whereas $z$ determines the permutation corresponding to the
vertical strands of the diagram. In view of Corollary \ref{cor:basis}, these
triples may be thought of as the abstract Brauer diagrams for any $M\in\ADE$.
For $M=\D_n$, there is a diagrammatic description of $\BrM(\D_n)$ in
\cite{CGWTangle}.
\end{remark}


\begin{thebibliography}{99}

\bibitem{BirWen} J.~S.~Birman, H.~Wenzl, {\it Braids, link
polynomials and a new algebra}, Trans.\ Amer.~Math.\ Soc., {\bf 313} (1989)
249--273.

\bibitem{Bou} {N.~Bourbaki,} {\it Groupes et alg\`ebres
de Lie, Chap 4, 5, et 6}, {Hermann, Paris,} {1968}.

\bibitem{Bra} {R.~Brauer,} {\it On algebras which are
connected with the semisimple continuous groups}, Annals of Math.,
{\bf 38} (1937) 857--872.
 
\bibitem{CGW} A.~M.~Cohen, D.~A.~H.~Gijsbers, and
D.~B.~Wales, {\it BMW algebras of simply laced type}, J.~Algebra, {\bf 286}
(2005) 107--153.

\bibitem{CGW2} A.~M.~Cohen, D.~A.~H.~Gijsbers, and
D.~B.~Wales, {\it A poset connected to Artin monoids of simply
laced type}, \rm 
Journal of Combinatorial Theory, Series A {\bf 113} (2006) 1646--1666.

\bibitem{CGWBMW} A.~M.~Cohen, D.~A.~H.~Gijsbers, and D.~B.~Wales,
{\it The BMW algebras of type $\D_n$},
\rm preprint, Eindhoven, 2007.


\bibitem{CGWTangle} A.~M.~Cohen, D.~A.~H.~Gijsbers, and
D.~B.~Wales, {\it Tangle and Brauer diagram algebras of type $\D_n$}, 
\rm preprint, Eindhoven, 2007.


\bibitem{Deo82} {V.~V.~Deodhar,} {\it On the root system of a Coxeter
group.} {Comm. Algebra}, {\bf 10} {(1982)} {611--630}.

\bibitem{Frenk}  B. Frenk, {\it Generalized Brauer algebras},
Master Thesis, Eindhoven, August 2006.

\bibitem{MorWas} H.~R.~Morton, A.~J.~Wasserman, {\it A
basis for the Birman-Wenzl Algebra}, \rm preprint, 1989,
\verb`liv.ac.uk/~su14/papers/WM.ps.gz`.


\bibitem{Mur} J.~Murakami, {\it The Kauffman polynomial of
links and representation theory}, Osaka J.~Math., {\bf 24} (1987)
745--758.

\bibitem{Wilf} M.~Petkovsek, H.S.~Wilf, D.~Zeilberger,
{\it $A=B$}, A.K.~Peters, Wellesley (MA), 1996.\\
\verb`http://www.cis.upennn.edu/~wilf`.

\bibitem{Wen} H.~Wenzl, {\it Quantum groups and subfactors of type B, C, and D}, Comm. Math. Phys., 
{\bf 133} {(1990)} {383--432}. 


\end{thebibliography}
\end{document}